\documentclass[11pt,reqno]{amsart}
\usepackage{amsmath, amssymb, amsthm}
\usepackage{url}

 \renewcommand{\a}{\alpha}
\renewcommand{\b}{\beta}
\newcommand{\e}{\epsilon}
\renewcommand{\d}{{\delta}}
\newcommand{\g}{\gamma}

\renewcommand{\(}{\left\(}
\renewcommand{\)}{\right\)}
\renewcommand{\[}{\left\[}
\renewcommand{\]}{\right\]}

\numberwithin{equation}{section}
 \theoremstyle{plain}
\newtheorem{theorem}{Theorem}[section]
\newtheorem{lemma}[theorem]{Lemma}

   \makeatletter
\def\proof{\@ifnextchar[{\@oproof}{\@nproof}}
\def\@oproof[#1][#2]{\trivlist\item[\hskip\labelsep\textit{#2 Proof of\
#1.}~]\ignorespaces}
\def\@nproof{\trivlist\item[\hskip\labelsep\textit{Proof.}~]\ignorespaces}

\makeatother

\begin{document}

\title[New representations for $\sigma(q)$]{New representations for $\sigma(q)$ via reciprocity theorems} 

\author{Koustav Banerjee}
\address{Department of Mathematics, Ramakrishna Mission Vivekananda University, PO Belur Math, Howrah,  711202, West Bengal, India} \email{banerjeekoustav98@gmail.com}

\author{Atul Dixit}
\address{Department of Mathematics, Indian Institute of Technology Gandhinagar, Palaj, Gandhinagar 382355, Gujarat, India} 
\email{adixit@iitgn.ac.in}
\dedicatory{\emph{Dedicated to our friend Krishnaswami Alladi on the occasion of his $60$th birthday}}
\maketitle

\begin{abstract}
Two new representations for Ramanujan's function $\sigma(q)$ are obtained. The proof of the first one uses the three-variable reciprocity theorem due to Soon-Yi Kang and a transformation due to R.P.~Agarwal while that of the second uses the four-variable reciprocity theorem due to George Andrews and a generalization of a recent transformation of Andrews, Schultz, Yee and the second author. The advantage of these representations is that they involve free complex parameters - one in the first representation, and two in the second.
\end{abstract}

\section{Introduction}\label{intro}
One of the celebrated functions of Ramanujan is the function $\sigma(q)$ defined by
\begin{equation*}
\sigma(q):=\sum_{n=0}^{\infty} \frac{q^{\frac{n(n+1)}{2}}}{(-q)_n}.
\end{equation*}
It is the generating function for the excess number of partitions of $n$ into distinct parts with even rank over those of odd rank \cite{andrews1986}. Note that the rank of a partition is the largest part minus the number of parts.
 
On page 14 of the Lost Notebook \cite{lnb}, Ramanujan gave two surprising identities involving $\sigma(q)$:
\begin{equation}\label{rsi1}
\sum_{n=0}^{\infty}\left(S(q)-(-q)_n\right)=S(q)D(q)+\frac{1}{2}\sigma(q),
\end{equation}
and
\begin{equation}\label{rsi2}
\sum_{n=0}^{\infty}\left(S(q)-\frac{1}{(q;q^2)_{n+1}}\right)=S(q)D(q^2)+\frac{1}{2}\sigma(q),
\end{equation}
where
\begin{align*}
S(q)&:=(-q;q)_{\infty},\nonumber\\
D(q)&=-\frac{1}{2}+\sum_{n=1}^{\infty}\frac{q^n}{1-q^n}.
\end{align*}
Here, and throughout the sequel, we assume $|q|<1$ and use the standard $q$-series notation
\begin{align*}
(A)_0 &:=(A;q)_0 =1, \qquad \\
(A)_n &:=(A;q)_n  = (1-A)(1-Aq)\cdots(1-Aq^{n-1})\hspace{2mm}\text{for any positive integer}\hspace{1mm} n, \\
(A)_{\infty} &:=(A;q)_{\infty}  = \lim_{n\to\infty}(A;q)_n, \qquad |q|<1,\\
(A)_{n}&:=(A)_{\infty}/(Aq^n)_{\infty}\hspace{2mm}\text{for any integer}\hspace{1mm} n.
\end{align*}
Since the base (or nome) of almost all of the $q$-shifted factorials occurring in our paper is $q$, for simplicity, we also use the following notation:
\begin{align*}
(A_1, A_2, \cdots, A_m)_n&:=(A_1, A_2, \cdots, A_m;q)_n=(A_1)_n(A_2)_n\cdots(A_m)_n,\\
(A_1, A_2, \cdots, A_m)_{\infty}&:=(A_1, A_2, \cdots, A_m;q)_{\infty}=(A_1)_{\infty}(A_2)_{\infty}\cdots(A_m)_{\infty}.
\end{align*}
Wherever there is a possibility of confusion, we provide the associated base.

The aforementioned identities involving $\sigma(q)$ were first proved by Andrews in \cite{andrews1986}. The function $\sigma(q)$ enjoys many nice properties relevant to various fields of number theory, namely, the theory of partitions, algebraic number theory, Maass waveforms, quantum modular forms etc. We review these properties below. 

In \cite{andrews1986}, and later more explicitly in \cite{andrewsmonthly86}, Andrews conjectured that infinitely many coefficients in the power series expansion of $\sigma(q)$ are zero but that the coefficients are unbounded. These two conjectures were later proved by Andrews, Dyson and Hickerson in a beautiful paper \cite{adh}, where they found that the coefficients of $\sigma(q)$ have multiplicative properties determined by a certain Hecke character associated to the real quadratic field $\mathbb{Q}(\sqrt{6})$. Results similar to these were later found by Bringmann and Kane \cite{bringkane1}, Corson, Favero, Liesinger and Zubairy \cite{cflz}, Lovejoy \cite{love1}, \cite{love2}, Patkowski \cite{pat0}, and more recently by Xiong \cite{xiong}.

Cohen \cite{cohen} showed that if we set 
\begin{align*}
\varphi(q)&:=q^{1/24}\sigma(q)+q^{-1/24}\sigma^{*}(q)\nonumber\\
&=\sum_{n\in\mathbb{Z}\atop n\equiv1\hspace{1mm}(\text{mod}\hspace{1mm}24)}T(n)q^{|n|/24},
\end{align*}
where
\begin{align*}
\sigma^{*}(q):=2\sum_{n=1}^{\infty}\frac{(-1)^nq^{n^2}}{(q;q^2)_n},
\end{align*}
then $T(n)$ are the coefficients of a Maass waveform of eigenvalue $1/4$. For another example of such a Maass waveform associated with the pair $(W_1(q), W_2(q))$ studied in \cite{cflz}, we refer the reader to Section $2$ of a recent paper of Li, Ngo and Rhoades \cite{lingorhoades1}. At the end of \cite{lingorhoades1}, the authors posed an open problem of relating $10$ other pairs of $q$-series to Maass waveforms or indefinite quadratic forms, which was recently solved by Krauel, Rolen and Woodbury \cite{krw}. The function $\sigma(q)$ also occurs in one of the first examples of quantum modular forms given by Zagier \cite{zagierqmf} that is, $q^{1/24}\sigma(q)$, where $q=e^{2\pi ix}, x\in\mathbb{Q}$, is a quantum modular form. 

The identities of the type \eqref{rsi1} and \eqref{rsi2} are known as `sum of tails' identities. After Ramanujan, Zagier \cite[Theorem 2]{zagiertop} was the next mathematician to discover a `sum of tails' identity. This is associated with the Dedekind eta-function and occurs in his work on Vassiliev invariants. Using a new Abel-type lemma, Andrews, Jim\'{e}nez-Urroz, and Ono \cite{ajo} obtained two general theorems involving $q$-series obtained by summing the iterated differences between an infinite product and its truncated products, and used them not only to prove \eqref{rsi1} and \eqref{rsi2} and similar other identities but also to determine the values at negative integers of certain $L$-functions. Chan \cite[p.~78]{chansears} gave a multiparameter `sum of tails' identity which consists, as special cases, the two general theorems in \cite{ajo}. More `sum of tails' identities were obtained by Andrews and Freitas \cite{af}, Bringmann and Kane \cite{bringkane2}, and Patkowski \cite{pat1}, \cite{pat2}, \cite{pat3}.

Andrews \cite{andrews1986} asked for a `near bjiection' between the weighted counts of partitions given by the left-hand sides of \eqref{rsi1} and \eqref{rsi2}, and the coefficients of the corresponding first expressions obtained by the convolutions of the associated partition functions. Such a proof was supplied by Chen and Ji \cite{chenji}. In \cite[Theorem 3.3]{ady1}, the function $\sigma(q)$ was found to be related to the generating function of the number of partitions of $n$ such that all even parts are less than or equal to twice the smallest part.

As mentioned before, identities \eqref{rsi1} and \eqref{rsi2} were proved by Andrews in \cite{andrews1986}. His proof was based on an application of a beautiful $q$-series identity of Ramanujan \cite[p.~40]{lnb}, \cite[Equation (3.8)]{gea90}, now known as Ramanujan's reciprocity theorem, which was in turn proved earlier by Andrews himself in \cite{gea90}. In \cite{af}, it was remarked that the proofs of \eqref{rsi1} and \eqref{rsi2} in \cite{andrews1986} are nearly as odd as the identities themselves. In \cite[p.~149]{ablnb2} as well, it was remarked that `the proofs provide no significant insight into the reasons for their existence'. While this may be true, the goal of this paper is to show that the underlying idea in these proofs  can be adapted to obtain new representations for $\sigma(q)$, which are of a type completely different than those previously known, for example, \cite[Equations (6.3), (6.4)]{adh} or \eqref{rsi1} and \eqref{rsi2}.

These representations result from applying Andrews' idea in \cite{andrews1986} to the three-variable reciprocity theorem of Kang \cite[Theorem 4.1]{kang} which is equivalent to Ramanujan's ${}_1\psi_{1}$ summation formula, and to the four-variable reciprocity theorem \cite[Theorem 1.2]{kang} which is equivalent to a formula of Andrews \cite[Theorem 6]{gea90}. 

For $|c|<|a|<1$ and $|c|<|b|<1$, Kang \cite[Theorem 4.1]{kang} obtained the following three-variable reciprocity theorem:
\begin{align}\label{3varrt}
\rho_{3}(a, b, c)-\rho_{3}(b, a, c)=\left(\frac{1}{b}-\frac{1}{a}\right)\frac{(c,aq/b, bq/a, q)_{\infty}}{(-c/a, -c/b, -aq, -bq)_{\infty}},
\end{align}
where 
\begin{align*}
\rho_{3}(a, b, c):=\left(1+\frac{1}{b}\right)\sum_{n=0}^{\infty}\frac{(c)_n(-1)^nq^{n(n+1)/2}a^nb^{-n}}{(-aq)_{n}(-c/b)_{n+1}}.
\end{align*}
Ramanujan's reciprocity theorem is a special case $c=0$ of the above theorem.

Using \eqref{3varrt}, we obtain the following new representation for $\sigma(q)$. The surprising thing about this representation is that it is valid for any complex $c$ such that $|c|<1$.
\begin{theorem}\label{newsigma1}
For any complex c such that $|c|<1$, we have
\begin{align}\label{idsigma1}
\sigma(q)=(-c)_{\infty}\sum_{n=0}^{\infty}\frac{q^{n(n+1)/2}}{(-q)_n(1-cq^n)}-2\sum_{m,n=0}^{\infty}\frac{(-q)_m}{(q)_m(q)_n}\frac{(-1)^nq^{n(n+1)/2}c^{m+n+1}}{(1-q^{n+m+1})}.
\end{align}
\end{theorem}
For $|c|, |d|<|a|, |b|<1$, the four-variable reciprocity theorem is given by \cite[Theorem 1.2]{kang}
\begin{align}\label{4varrt}
\rho_{4}(a,b,c,d)-\rho_4(b, a, c, d)=\left(\frac{1}{b}-\frac{1}{a}\right)\frac{(d,c,cd/(ab),aq/b, bq/a, q)_{\infty}}{(-d/a, -d/b, -c/a, -c/b, -aq, -bq)_{\infty}},
\end{align}
where
\begin{equation*}
\rho_{4}(a,b,c,d):=\left(1+\frac{1}{b}\right)\sum_{n=0}^{\infty}\frac{(d, c, cd/(ab))_n(1+cdq^{2n}/b)(-1)^nq^{n(n+1)/2}a^nb^{-n}}{(-aq)_n(-c/b, -d/b)_{n+1}}.
\end{equation*}
Using \eqref{4varrt}, we obtain the following new representation for $\sigma(q)$ which consists of two free complex parameters $c, d$:
\begin{theorem}\label{newsigma2}
Let $|c|<1$ and $|d|<1$. 
Then
\begin{align}\label{idsigma2}
\sigma(q)=\frac{(-c,-d)_{\infty}}{(-cd)_{\infty}}\sum_{n=0}^{\infty}\frac{(-cd)_n(1-cdq^{2n})q^{n(n+1)/2}}{(-q)_n(1-cq^n)(1-dq^n)}+\Lambda(c, d, q),
\end{align}
where
{\allowdisplaybreaks\begin{align}\label{Lambda}
\Lambda(c, d, q)&=1-\frac{3cd(-cq,-dq)_{\infty}}{(-cd,-q)_{\infty}}\sum_{n=0}^{\infty}\frac{(c,d,-cd)_n}{(-cq,-dq,q)_n}q^{\frac{n(n+1)}{2}+2n}\nonumber\\
&\quad-\frac{(-dq,c)_{\infty}}{(-cd,-q)_{\infty}}\sum_{p=0}^{\infty}\frac{(-q)_p(-1)_pc^p}{(q)_p}\sum_{k=0}^{\infty}\frac{(-q^{p+1})_kc^k}{(q^{p+1})_k}\nonumber\\
&\quad\quad\times\sum_{n=0}^{\infty}\frac{(-cd,d)_n}{(-dq,q)_n}q^{\frac{n(n+1)}{2}+(p+k)n}(1+cdq^{2n}(1+q^{p})(1+q^{p+1}))\nonumber\\
&\quad-\frac{(d,c)_{\infty}}{(-cd,-q)_{\infty}}\sum_{p=1}^{\infty}\frac{(-q)_p(-1)_pd^p}{(q)_p}\sum_{k=0}^{\infty}\frac{(-q)_k(-q^p)_kc^k}{(q)_k}\sum_{j=0}^{\infty}\frac{(-q^{p+1})_jd^j}{(q^{p+1})_j}\nonumber\\
&\quad\quad\times\sum_{n=0}^{\infty}\frac{(-cd)_n}{(q)_n}q^{\frac{n(n+1)}{2}+(p+k+j)n}(1+cdq^{2n}(1+q^{p+k})(1+q^{p+k+1}))\nonumber\\
&\quad-2(d,c)_{\infty}\sum_{p=1}^{\infty}\frac{(-cd)^{p}}{(q)_p}\sum_{j=0}^{\infty}\frac{(-q)_jd^j}{(q)_j}\sum_{k=0}^{\infty}\frac{(-q)_kc^k}{(q)_k}\nonumber\\
&\quad\quad\times\sum_{m=0}^{\infty}\frac{(-cd)^m}{(q^{p+1})_m(-q^{p+k+j})_{m+1}}\left(1+cd\frac{(1+q^{p+k+j})(1+q^{p+k+j+1})}{(1+q^{p+k+j+m+1})(1+q^{p+k+j+m+2})}\right).
\end{align}}
\end{theorem}
It will be shown later that letting $d=0$ in Theorem \ref{newsigma2} results in Theorem \ref{newsigma1}. Still, pedagogically it is sound to first give a proof of Theorem \ref{newsigma1} and then proceed to that of Theorem \ref{newsigma2}, especially since the complexity involved in the former is much lesser than that in the latter.

In order to obtain \eqref{idsigma2}, we derive a new nine-parameter transformation contained in the following theorem which generalizes previous transformations due to Agarwal \cite{agar1} (see Equation \eqref{me} below), and due to Andrews, Dixit, Schultz and Yee \cite{adsy1} (see Equation \ref{extandagar} below).
\begin{theorem}\label{extandagar9}
For $\b, \d, f, h, t\neq q^{-j}, j\geq 0$, the following identity is true:
{\allowdisplaybreaks\begin{align}\label{extandagarfur}
&\sum_{n=0}^{\infty}\frac{(\a)_{n}(\g)_{n}(e)_n(g)_n}{(\b)_{n}(\d)_{n}(f)_{n}(h)_n}t^n\nonumber\\
&=\frac{(g, e, \g,\frac{\b}{\a}, q, \a t, \frac{q}{\a t}, \frac{\d q}{\b}, \frac{fq}{\b}, \frac{hq}{\b})_{\infty}}{(h, f, \d, \frac{q}{\a}, \b, \frac{\b}{\a t}, \frac{\a tq}{\b}, \frac{\g q}{\b}, \frac{e q}{\b}, \frac{gq}{\b})_{\infty}}{}_{4}\phi_{3}\left(\begin{matrix}\frac{\a q}{\b},& \frac{\g q}{\b},&\frac{e q}{\b},&\frac{gq}{\b}\\
&\frac{\d q}{\b},& \frac{fq}{\b},& \frac{hq}{\b}\end{matrix}\, ;q, t\right)\nonumber\\
&\quad+\bigg(1-\frac{q}{\b}\bigg)\frac{(g,e, \g, t, \frac{\d q}{\b}, \frac{fq}{\b}, \frac{hq}{\b})_{\infty}}{(h, f, \d, \frac{\a t}{\b}, \frac{\g q}{\b}, \frac{e q}{\b}, \frac{gq}{\b})_{\infty}}{}_{4}\phi_{3}\left(\begin{matrix}\frac{\a q}{\b},& \frac{\g q}{\b},&\frac{e q}{\b},&\frac{gq}{\b}\\
&\frac{\d q}{\b},& \frac{fq}{\b},& \frac{hq}{\b}\end{matrix}\, ;q, t\right)
\bigg({}_{2}\phi_{1}\bigg(\begin{matrix}q,& \frac{q}{t}\\
&\frac{\b q}{\a t}\end{matrix}\, ;\frac{q}{\a}\bigg)-1\bigg)\nonumber\\
&\quad+\bigg(1-\frac{q}{\b}\bigg)\frac{(g, e, \g, t, \frac{fq}{\b}, \frac{hq}{\b})_{\infty}}{(h, f, \d, \frac{\a t}{\b}, \frac{e q}{\b}, \frac{gq}{\b})_{\infty}}\sum_{p=0}^{\infty}\frac{(\frac{\d}{\g})_p(\frac{\a t}{\b})_p\g^p}{(t)_p(q)_p}\sum_{k=0}^{\infty}\frac{(\frac{\d q^p}{\g})_k(\frac{\g q}{\b})^k}{(q^{p+1})_k}{}_{3}\phi_{2}\bigg(\begin{matrix}\frac{\a q}{b},& \frac{e q}{\b},&\frac{gq}{\b}\\
&\frac{fq}{\b},&\frac{hq}{\b}\end{matrix}\, ;q, tq^{p+k}\bigg)\nonumber\\
&\quad+\bigg(1-\frac{q}{\b}\bigg)\frac{(g, e, \g, t, \frac{hq}{\b})_{\infty}}{(h, f, \d, \frac{\a t}{\b}, \frac{gq}{\b})_{\infty}}\nonumber\\
&\quad\quad\times\sum_{p=1}^{\infty}\frac{(\frac{f}{e})_p(\frac{\a t}{\b})_pe^p}{(t)_p(q)_p}\sum_{k=0}^{\infty}\frac{(\frac{\d}{\g})_k(\frac{\a tq^p}{\b})_k\g^k}{(q)_k(tq^{p})_k}\sum_{j=0}^{\infty}\frac{(\frac{fq^p}{e})_j(\frac{e q}{\b})^j}{(q^{p+1})_j}{}_{2}\phi_{1}\bigg(\begin{matrix}\frac{\a q}{b},&\frac{gq}{\b}\\
&\frac{hq}{\b}\end{matrix}\, ;q, tq^{p+k+j}\bigg)\nonumber\\
&\quad+\bigg(1-\frac{q}{\b}\bigg)\frac{(g,e, \g)_{\infty}}{(h, f, \d)_{\infty}}\sum_{p=1}^{\infty}\frac{(\frac{h}{g})_pg^p}{(q)_p}\sum_{j=0}^{\infty}\frac{(\frac{f}{e})_je^j}{(q)_j}\sum_{k=0}^{\infty}\frac{(\frac{\d}{\g})_k\g^k}{(q)_k}\sum_{m=0}^{\infty}\frac{(\frac{hq^p}{g})_m(tq^{p+k+j})_m}{(q^{p+1})_m(\frac{\a tq^{p+k+j}}{\b})_{m+1}}\left(\frac{gq}{\b}\right)^{m}.
\end{align}}
\end{theorem}
A version of the above formula, and also of \eqref{extandagar}, in terms of $q$-Lauricella functions, was obtained by Gupta \cite[p.~53]{mgupta} in his PhD thesis. However, his versions are not as explicit as the ones in \eqref{extandagarfur} and \eqref{extandagar}. We remark that Gupta has obtained a general transformation of these results, with $r$ $q$-shifted factorials in the numerator and $r$ in the denominator, in terms of $q$-Lauricella functions. However, one can easily anticipate such general transformation by observing the pattern occurring in Agarwal's identity \eqref{me}, \eqref{extandagar} and \eqref{extandagarfur}. To avoid digression, we do not pursue it here.

This paper is organized as follows. In Section \ref{prelim}, we collect formulas from the literature that are used in the sequel. Section \ref{3var} is devoted to proving Theorem \ref{newsigma1} while Section \ref{4var} to proving Theorem \ref{newsigma2}, and for deriving Theorem \ref{newsigma1} from Theorem \ref{newsigma2}. We conclude this paper with Section \ref{cr} consisting of some remarks and directions for further research.
\section{Preliminaries}\label{prelim}
The $q$-binomial theorem \cite[p.~17, Equation (2.2.1)]{gea} states that for $|z|<1$,
\begin{equation}\label{qbin}
\sum_{n=0}^{\infty}\frac{(a;q)_{n}z^n}{(q;q)_n}=\frac{(az;q)_{\infty}}{(z;q)_{\infty}}.
\end{equation}

For $|z|<1$ and $|b|<1$, Heine's transformation \cite[p.~38]{gea} is given by
\begin{equation}\label{heine}
{}_2\phi_{1}\left(\begin{matrix}a,& b\\
&c\end{matrix}\, ;q, t\right)=\frac{(b, at)_{\infty}}{(c, t)_{\infty}}{}_2\phi_{1}\left(\begin{matrix}c/b,& t\\
&at\end{matrix}\, ;q, b\right),
\end{equation}
where as its second iterate \cite[p.~38, \text{last line}]{gea} is
\begin{equation}\label{heine2}
{}_2\phi_{1}\left(\begin{matrix}a,& b\\
&c\end{matrix}\, ;q, t\right)=\frac{(c/b, bt)_{\infty}}{(c, t)_{\infty}}{}_2\phi_{1}\left(\begin{matrix}b,& abt/c\\
&bt\end{matrix}\, ;q, c/b\right).
\end{equation}
Here ${}_{r+1}\phi_{r}$ is the basic hypergeometric series defined by
\begin{equation*}\label{bhs}
{}_{r+1}\phi_{r}\left(\begin{matrix} a_1, a_2, \ldots, a_{r+1}\\
  b_1,  b_2, \ldots, b_{r} \end{matrix}\,; q,
z \right) :=\sum_{n=0}^{\infty} \frac{(a_1;q)_n (a_2;q)_n \cdots (a_r;q)_n}{(q;q)_n (b_1;q)_n \cdots (b_{s};q)_n} z^n.
\end{equation*}
We need the following identity \cite[p.~17, Equation (15.51)]{fine}:
\begin{equation}\label{1551fine}
\sum_{n=0}^{\infty}\frac{t^n}{(bq)_n}=\frac{(1-b)}{(t)_{\infty}}\sum_{n=0}^{\infty}\frac{(-t)^nq^{n(n+1)/2}}{(q)_n(1-bq^n)}.
\end{equation}
Agarwal \cite[Equation (3.1)]{agar1} obtained the following `mild' extension/generalization of an important identity of Andrews \cite[Theorem 1]{gea90} in the sense that we get Andrews' identity from the following result when $t=q$:
\begin{align}\label{me}
&\sum_{n=0}^{\infty}\frac{(\a)_{n}(\g)_{n}}{(\b)_{n}(\d)_{n}}t^n\nonumber\\
&=\frac{(q/(\a t), \g, \a t, \b/\a, q)_{\infty}}{(\b/(\a t), \d, t, q/\a, \b)_{\infty}}{}_{2}\phi_{1}\bigg(\begin{matrix}\d/\g,& t\\
&q\a t/\b\end{matrix}\, ;q, \g q/\b\bigg)\nonumber\\
&\quad+\frac{(\g)_{\infty}}{(\d)_{\infty}}\left(1-\frac{q}{\b}\right)\sum_{m=0}^{\infty}\frac{(\d/\g)_{m}(t)_m}{(q)_m(\a t/\b)_{m+1}}(q\g/\b)^m\left({}_{2}\phi_{1}\bigg(\begin{matrix}q,& q/t\\
&q\b/(\a t)\end{matrix}\, ;q, q/\a\bigg)-1\right)\nonumber\\
&\quad+\frac{(\g)_{\infty}}{(\d)_{\infty}}\left(1-\frac{q}{\b}\right)\sum_{p=0}^{\infty}\frac{\g^p(\d/\g)_p}{(q)_p}\sum_{m=0}^{\infty}\frac{(\d q^p/\g)_m(tq^p)_m}{(q^{1+p})_m(\a tq^p/\b)_{m+1}}(q\g/\b)^m.
\end{align}
The following generalization of the above identity of Agarwal was recently obtained in \cite[Theorem 3.1]{adsy1} for $\b,\d, f, t\neq q^{-j}, j\geq 0$:
{\allowdisplaybreaks\begin{align}\label{extandagar}
&\sum_{n=0}^{\infty}\frac{(\a)_{n}(\g)_{n}(e)_n}{(\b)_{n}(\d)_{n}(f)_{n}}t^n\nonumber\\
&=\frac{(e, \g, \frac{\b}{\a}, q, \a t, \frac{q}{\a t}, \frac{\d q}{\b}, \frac{fq}{\b})_{\infty}}{(f, \d, \frac{q}{\a}, \b, \frac{\b}{\a t}, \frac{\a tq}{\b}, \frac{\g q}{\b}, \frac{e q}{\b})_{\infty}}{}_{3}\phi_{2}\left(\begin{matrix}\frac{\a q}{\b},& \frac{\g q}{\b},&\frac{e q}{\b}\\
&\frac{\d q}{\b},& \frac{fq}{\b}\end{matrix}\, ;q, t\right)\nonumber\\
&\quad+\bigg(1-\frac{q}{\b}\bigg)\frac{(e, \g, t, \frac{\d q}{\b}, \frac{fq}{\b})_{\infty}}{(f, \d, \frac{\a t}{\b}, \frac{\g q}{\b}, \frac{e q}{\b})_{\infty}}{}_{3}\phi_{2}\bigg(\begin{matrix}\frac{\a q}{\b},& \frac{\g q}{\b},&\frac{e q}{\b}\\
&\frac{\d q}{\b},& \frac{fq}{\b}\end{matrix}\, ;q, t\bigg)\bigg({}_{2}\phi_{1}\bigg(\begin{matrix}q,& \frac{q}{t}\\
&\frac{\b q}{\a t}\end{matrix}\, ;\frac{q}{\a}\bigg)-1\bigg)\nonumber\\
&\quad+\bigg(1-\frac{q}{\b}\bigg)\frac{(e, \g, t, \frac{fq}{\b})_{\infty}}{(f, \d, \frac{\a t}{\b}, \frac{e q}{\b})_{\infty}}\sum_{p=0}^{\infty}\frac{(\frac{\d}{\g})_p(\frac{\a t}{\b})_p\g^p}{(t)_p(q)_p}\sum_{k=0}^{\infty}\frac{(\frac{\d q^p}{\g})_k(\frac{q\g}{\b})^k}{(q^{1+p})_k}{}_{2}\phi_{1}\bigg(\begin{matrix}\frac{\a q}{b},& \frac{e q}{\b}\\
&\frac{fq}{\b}\end{matrix}\, ;q, tq^{k+p}\bigg)\nonumber\\
&\quad+\bigg(1-\frac{q}{\b}\bigg)\frac{(e, \g)_{\infty}}{(f, \d)_{\infty}}\sum_{p=1}^{\infty}\frac{(\frac{f}{e})_pe^p}{(q)_p}\sum_{k=0}^{\infty}\frac{(\frac{\d}{\g})_k\g^k}{(q)_k}\sum_{m=0}^{\infty}\frac{(\frac{fq^p}{e})_m(tq^{p+k})_m}{(q^{1+p})_m(\frac{\a tq^{p+k}}{\b})_{m+1}}\left(\frac{e q}{\b}\right)^{m}.
\end{align}}
We will also make use of the $\e$-operator acting on a differentiable function $f$ by \cite{andrews1986}
\begin{equation*}
\e(f(z))=f'(1).
\end{equation*}
\section{The three-variable case}\label{3var}
We prove Theorem \ref{newsigma1} here. Letting $a=-z$ and $b=1$ in \eqref{3varrt} gives
\begin{equation}\label{3varrtspel}
\rho_3(1,-z,c)=\rho_{3}(-z,1,c)-\frac{(c,-zq,-z^{-1},q)_{\infty}}{(cz^{-1},-c,zq,-q)_{\infty}}.
\end{equation}
Divide both sides by $(1-z^{-1})$ and let $z\to 1$. It is easy to see that the left side becomes  $\displaystyle\sum_{n=0}^{\infty}\frac{q^{n(n+1)/2}}{(-q)_n(1-cq^n)}$, which we denote by $\sigma(c, q)$. Denote the right-hand side of the above equation by $f(z)$. Note that 
\begin{align*}
\lim_{z\to 1}f(z)=2\sum_{n=0}^{\infty}\frac{(c)_nq^{n(n+1)/2}}{(q)_n(-c)_{n+1}}-2\frac{(-q)_{\infty}}{(-c)_{\infty}}=0, 
\end{align*}
since replacing $c$ by $-cq$, substituting $a=c, b=-q/\tau, t=\tau$, and then letting $\tau\to 0$ in \eqref{heine2} gives 
\begin{align*}
&\sum_{n=0}^{\infty}\frac{(c)_nq^{n(n+1)/2}}{(q)_n(-cq)_{n}}=\lim_{\tau\to 0}{}_2\phi_{1}\left(\begin{matrix}c,& -q/\tau\\
&-cq\end{matrix}\, ;q, \tau\right)=\lim_{\tau\to 0}\frac{(c\tau, -q)_{\infty}}{(-cq, \tau)_{\infty}}{}_2\phi_{1}\left(\begin{matrix}-q/\tau,& 1\\
&-q\end{matrix}\, ;q, c\tau\right)\nonumber\\
&=\lim_{\tau\to 0}\frac{(c\tau,-q)_{\infty}}{(-cq,\tau)_{\infty}}=\frac{(-q)_{\infty}}{(-cq)_{\infty}}.
\end{align*}
This result can also be found in \cite[Corollary 7.5]{kang}.

Hence using L'Hopital's rule, we see that
\begin{align*}
\lim_{z\to 1}\frac{f(z)}{1-z^{-1}}=f'(1)=\e(f(z)),
\end{align*}
so that from \eqref{3varrtspel},
\begin{align}\label{rcq}
\sum_{n=0}^{\infty}\frac{q^{n(n+1)/2}}{(-q)_n(1-cq^n)}=\e\left(\rho_{3}(-z,1,c)-\frac{(c,-zq,-z^{-1},q)_{\infty}}{(cz^{-1},-c,zq,-q)_{\infty}}\right).
\end{align}

The idea now is to rightly transform $\rho_3(-z,1,c)=2\displaystyle\sum_{n=0}^{\infty}\frac{(c)_nz^nq^{n(n+1)/2}}{(zq)_n(-c)_{n+1}}$ into an expression which is amenable to the $\e$-operator. To that end, we invoke \eqref{me}, the reasons for which will be clear soon. Let $\a=-q/\tau, \g=c, \b=zq, \d=-cq$ and $t=\tau z$ in \eqref{me}. Then,
\begin{align*}
&\sum_{n=0}^{\infty}\frac{(-q/\tau)_n(c)_n}{(zq)_n(-cq)_n}(\tau z)^n\nonumber\\
&=\frac{(-z^{-1}, c, -zq, -\tau z, q)_{\infty}}{(-1, -cq, \tau z, -\tau, zq)_{\infty}}\sum_{m=0}^{\infty}\frac{(\tau z)_m}{(q)_m}\left(\frac{c}{z}\right)^m+\frac{(c)_{\infty}}{2(-cq)_{\infty}}\left(1-\frac{1}{z}\right)\sum_{m=0}^{\infty}\frac{(\tau z)_m(\frac{c}{z})^m}{(q)_m}\sum_{j=1}^{\infty}\frac{\left(\frac{q}{\tau z}\right)_j}{(-q)_j}(-\tau)^j\nonumber\\
&\quad+\frac{(c)_{\infty}}{(-cq)_{\infty}}\left(1-\frac{1}{z}\right)\sum_{p=0}^{\infty}\frac{(-q)_pc^p}{(q)_p}\sum_{m=0}^{\infty}\frac{(\tau zq^p)_m}{(q^{p+1})_m}\frac{(c/z)^m}{(1+q^p)}.
\end{align*} 
Now use \eqref{qbin} to evaluate sums over $m$ in the first two expressions, then let $\tau\to 0$ on both sides, separate the term corresponding to $p=0$ in the double sum followed by another application of \eqref{qbin}, and finally multiply throughout by $2/(1+c)$ to obtain
\begin{align}\label{aftagar}
&2\sum_{n=0}^{\infty}\frac{(c)_nz^nq^{n(n+1)/2}}{(zq)_n(-c)_{n+1}}\nonumber\\
&=\frac{(-z^{-1}, c, -zq, q)_{\infty}}{(-q, -c, zq,c/z)_{\infty}}+\frac{(c)_{\infty}}{(-c,c/z)_{\infty}}\left(1-\frac{1}{z}\right)\sum_{j=1}^{\infty}\frac{q^{j(j+1)/2}z^{-j}}{(-q)_j}\nonumber\\
&\quad+\frac{(c)_{\infty}}{(-c,c/z)_{\infty}}\left(1-\frac{1}{z}\right)+2\frac{(c)_{\infty}}{(-c)_{\infty}}\left(1-\frac{1}{z}\right)\sum_{p=1}^{\infty}\frac{(-q)_pc^p}{(q)_p(1+q^p)}\sum_{m=0}^{\infty}\frac{(c/z)^m}{(q^{p+1})_m}\nonumber\\
&=\frac{(-z^{-1}, c, -zq, q)_{\infty}}{(-q, -c, zq,c/z)_{\infty}}+\frac{(c)_{\infty}}{(-c,c/z)_{\infty}}\left(1-\frac{1}{z}\right)\sum_{j=1}^{\infty}\frac{q^{j(j+1)/2}z^{-j}}{(-q)_j}\nonumber\\
&\quad+\frac{(c)_{\infty}}{(-c,c/z)_{\infty}}\left(1-\frac{1}{z}\right)+2\frac{(c)_{\infty}}{(-c, c/z)_{\infty}}\left(1-\frac{1}{z}\right)\sum_{p=1}^{\infty}\frac{(-q)_{p-1}c^p}{(q)_{p-1}}\sum_{n=0}^{\infty}\frac{(-c/z)^nq^{n(n+1)/2}}{(q)_n(1-q^{n+p})},
\end{align}
where in the last step we applied \eqref{1551fine} with $t=c/z$ and $b=q^{p}$. 

Now substitute \eqref{aftagar} in \eqref{rcq} and then apply the $\e$-operator to deduce that
\begin{align*}
&\sum_{n=0}^{\infty}\frac{q^{n(n+1)/2}}{(-q)_n(1-cq^n)}\nonumber\\
&=\frac{(\sigma(q)-1)}{(-c)_{\infty}}+\frac{1}{(-c)_{\infty}}+\frac{2}{(-c)_{\infty}}\sum_{p=1}^{\infty}\frac{(-q)_{p-1}c^p}{(q)_{p-1}}\sum_{n=0}^{\infty}\frac{(-c)^nq^{n(n+1)/2}}{(q)_n(1-q^{n+p})},
\end{align*}
which is nothing but \eqref{idsigma1}. This completes the proof.

\textbf{Remark 1.} If we explicitly evaluate $\e\left(\frac{(c,-zq,-z^{-1},q)_{\infty}}{(cz^{-1},-c,zq,-q)_{\infty}}\right)$ using the Jacobi triple product identity \cite[p.~28, Theorem 2.8]{gea}, then, from \eqref{rcq}, we obtain upon simplification
\begin{equation*}
\sigma(c, q)=2\e\left(\sum_{n=0}^{\infty}\frac{(c)_nz^nq^{n(n+1)/2}}{(zq)_n(-c)_{n+1}}\right)+S(c, q)+2S(c, q)\left(\sum_{n=0}^{\infty}\frac{cq^n}{1-cq^n}-\sum_{n=1}^{\infty}\frac{q^n}{1-q^n}\right),
\end{equation*}
where $S(c, q)=(-q)_{\infty}/(-c)_{\infty}$. This is a one-variable generalization of \cite[Equation (3.5)]{andrews1986}, as can be easily seen with the help of \eqref{rsi1}.
\section{The four-variable case}\label{4var}
We begin with a lemma that is used several times in the sequel.
\begin{lemma}\label{4heine}
For $|c|<1, |d|<1$, we have
\begin{align*}
\sum_{n=0}^{\infty}\frac{(c, d, -cd)_n}{(-cq,-dq,q)_n}(1+cdq^{2n})q^{n(n+1)/2}=\frac{(-cd,-q)_{\infty}}{(-cq,-dq)_{\infty}}.
\end{align*}
\end{lemma}
\begin{proof}
By Proposition 8 in \cite{chuzhang}, which is, in fact, equal to \eqref{4varrt}, we see that
\begin{align*}
&y\sum_{n=0}^{\infty}(1-q^{2n+1}y/x)\frac{\left(\frac{q}{bx},\frac{q}{cx},\frac{q}{dx}\right)_n}{(by, dy)_{n+1}(cyq)_n}(-bcdxy^2/q)^nq^{n(n+1)/2}\nonumber\\
&-x(1-cy)\sum_{n=0}^{\infty}(1-q^{2n+1}x/y)\frac{\left(\frac{q}{by},\frac{q}{cy},\frac{q}{dy}\right)_n}{(bx, cx, dx)_{n+1}}(-bcdx^2y/q)^nq^{n(n+1)/2}\nonumber\\
&=(y-x)\frac{(q,\frac{qy}{x},\frac{qx}{y}, bcxy, cdxy, bdxy)_{\infty}}{(bx,by,cx, cyq, dx, dy)_{\infty}}.
\end{align*}
Now let $d=q/(ux), b=q/(vx), c=1/y$ and $y=-uvx/q$ in the above identity to obtain upon simplification
\begin{align*}
\sum_{n=0}^{\infty}\frac{(u, v, -uv)_n}{(-uq,-vq,q)_n}(1+uvq^{2n})q^{n(n+1)/2}=\frac{(-uv,-q)_{\infty}}{(-uq,-vq)_{\infty}}.
\end{align*}
This completes the proof.
\end{proof}
We now first prove Theorem \ref{extandagar9} and then use it to prove Theorem \ref{newsigma2}. Since the underlying idea in the proof of Theorem \eqref{extandagar9} is similar to that involved in the proof of \eqref{me} (see \cite{agar1}) and in the proof of \eqref{extandagar} (see \cite{adsy1}), we will be very brief here. 

Let $S$ denote the sum on the left side of \eqref{extandagarfur}. Writing $(g)_n/(h)_n=((g)_{\infty}/(h)_{\infty})\cdot((hq^n)_{\infty}/(gq^n)_{\infty})$, then representing $((hq^n)_{\infty}/(gq^n)_{\infty})$ as a sum using \eqref{qbin}, interchanging the order of summation, and then employing \eqref{extandagar}, we find that
\begin{align}\label{S}
S=X+Y,
\end{align}
where
\begin{align*}
X&:=\frac{(g,e,\frac{\b}{\a},q,\frac{\d q}{\b},\frac{fq}{\b})_{\infty}}{(h,f,\d,\frac{q}{\a},\b,\frac{\g q}{\b},\frac{eq}{\b})_{\infty}}\sum_{m=0}^{\infty}\frac{(\frac{h}{g})_m(\a tq^m,\frac{q^{1-m}}{\a t})_{\infty}g^m}{(q)_m(\frac{\b q^{-m}}{\a t},\frac{\a tq^{m+1}}{\b})_{\infty}}{}_3\phi_{2}\bigg(\begin{matrix}\frac{\a q}{\b},& \frac{\g q}{\b},&\frac{e q}{\b}\\
&\frac{\d q}{\b},& \frac{fq}{\b}\end{matrix}\, ;q, tq^m\bigg),\nonumber\\
Y&:=\frac{(g,e,\g)_{\infty}(1-\frac{q}{\b})}{(h,f,\d)_{\infty}}\sum_{m=0}^{\infty}\frac{(\frac{h}{g})_mg^m}{(q)_m}\sum_{j=0}^{\infty}\frac{(\frac{f}{e})_je^j}{(q)_j(1-\frac{\a tq^{m+j}}{\b})}\sum_{k=0}^{\infty}\frac{(\frac{\d}{\g},\frac{\a tq^{m+j}}{\b})_k\g^k}{(q,\frac{\a tq^{j+1+m}}{\b})_k}\sum_{r=0}^{\infty}\frac{(\frac{q^{1-k-j-m}}{t})_r(\frac{q}{\a})^r}{(\frac{\b q^{1-k-j-m}}{\a t})_r}.
\end{align*}
To evaluate $X$, we write the ${}_3\phi_{2}$ in the form of series, interchange the order of summation, make use of the identity $(\b q^{-m}/(\a t))_{\infty}=(-\b/(\a t))^mq^{-m(m+1)/2}(\b/(\a t))_{\infty}(\a tq/\b)_m$, and then use \eqref{qbin} again to deduce
\begin{equation}\label{xeval}
X=\frac{(g, e, \g,\frac{\b}{\a}, q, \a t, \frac{q}{\a t}, \frac{\d q}{\b}, \frac{fq}{\b}, \frac{hq}{\b})_{\infty}}{(h, f, \d, \frac{q}{\a}, \b, \frac{\b}{\a t}, \frac{\a tq}{\b}, \frac{\g q}{\b}, \frac{e q}{\b}, \frac{gq}{\b})_{\infty}}{}_{4}\phi_{3}\left(\begin{matrix}\frac{\a q}{\b},& \frac{\g q}{\b},&\frac{e q}{\b},&\frac{gq}{\b}\\
&\frac{\d q}{\b},& \frac{fq}{\b},& \frac{hq}{\b}\end{matrix}\, ;q, t\right).
\end{equation}
Since
\begin{align}\label{v2intg}
\sum_{r=0}^{\infty}\frac{(\frac{q^{1-k-m-j}}{t})_r}{(\frac{\b q^{1-k-m-j}}{\a t})_r}\left(\frac{q}{\a}\right)^r
&=\frac{(t)_{m+j+k}}{\left(\frac{\a t}{\b}\right)_{m+j+k}}\left(\frac{q}{\b}\right)^{m+j+k}\bigg(\sum_{p=1}^{\infty}\frac{\left(\frac{q}{t}\right)_p}{\left(\frac{\b q}{\a t}\right)_p}\left(\frac{q}{\a}\right)^p+\sum_{p=0}^{m+j+k}\frac{(\frac{\a t}{\b})_p}{(t)_p}\left(\frac{\b}{q}\right)^p\bigg),
\end{align}
we observe that
\begin{align}\label{y12}
Y=Y_1+Y_2,
\end{align}
where $Y_1$ is associated with the infinite sum on the right of \eqref{v2intg} and $Y_2$ is associated with the finite sum. Even though the calculations for evaluating $Y_1$ and $Y_2$ are quite tedious, they are fairly straightforward. Using \eqref{heine}, it can be seen that
\begin{align}\label{y1}
Y_1&=\bigg(1-\frac{q}{\b}\bigg)\frac{(g,e, \g, t, \frac{\d q}{\b}, \frac{fq}{\b}, \frac{hq}{\b})_{\infty}}{(h, f, \d, \frac{\a t}{\b}, \frac{\g q}{\b}, \frac{e q}{\b}, \frac{gq}{\b})_{\infty}}{}_{4}\phi_{3}\left(\begin{matrix}\frac{\a q}{\b},& \frac{\g q}{\b},&\frac{e q}{\b},&\frac{gq}{\b}\\
&\frac{\d q}{\b},& \frac{fq}{\b},& \frac{hq}{\b}\end{matrix}\, ;q, t\right)
\bigg({}_{2}\phi_{1}\bigg(\begin{matrix}q,& \frac{q}{t}\\
&\frac{\b q}{\a t}\end{matrix}\, ;\frac{q}{\a}\bigg)-1\bigg).
\end{align}
Now write the finite sum on $p$ in $Y_2$ as
\begin{equation*}
\sum_{p=0}^{m+j+k}=\sum_{p=0}^{k}+\sum_{p=k+1}^{k+j}+\sum_{p=k+j+1}^{m+j+k},
\end{equation*}
and let $Y_{21}, Y_{22}$ and $Y_{23}$ denote the expressions associated with the first, second and third finite sums in the above equation respectively so that
\begin{equation}\label{y2123}
Y_2=Y_{21}+Y_{22}+Y_{23}.
\end{equation}
Now using \eqref{heine} repeatedly, it can be see that 
\begin{align}\label{y2123eval}
Y_{21}&=\bigg(1-\frac{q}{\b}\bigg)\frac{(g, e, \g, t, \frac{fq}{\b}, \frac{hq}{\b})_{\infty}}{(h, f, \d, \frac{\a t}{\b}, \frac{e q}{\b}, \frac{gq}{\b})_{\infty}}\sum_{p=0}^{\infty}\frac{(\frac{\d}{\g})_p(\frac{\a t}{\b})_p\g^p}{(t)_p(q)_p}\sum_{k=0}^{\infty}\frac{(\frac{\d q^p}{\g})_k(\frac{\g q}{\b})^k}{(q^{p+1})_k}{}_{3}\phi_{2}\bigg(\begin{matrix}\frac{\a q}{b},& \frac{e q}{\b},&\frac{gq}{\b}\\
&\frac{fq}{\b},&\frac{hq}{\b}\end{matrix}\, ;q, tq^{p+k}\bigg),\nonumber\\
Y_{22}&=\bigg(1-\frac{q}{\b}\bigg)\frac{(g, e, \g, t, \frac{hq}{\b})_{\infty}}{(h, f, \d, \frac{\a t}{\b}, \frac{gq}{\b})_{\infty}}\nonumber\\
&\quad\quad\times\sum_{p=1}^{\infty}\frac{(\frac{f}{e})_p(\frac{\a t}{\b})_pe^p}{(t)_p(q)_p}\sum_{k=0}^{\infty}\frac{(\frac{\d}{\g})_k(\frac{\a tq^p}{\b})_k\g^k}{(q)_k(tq^{p})_k}\sum_{j=0}^{\infty}\frac{(\frac{fq^p}{e})_j(\frac{e q}{\b})^j}{(q^{p+1})_j}{}_{2}\phi_{1}\bigg(\begin{matrix}\frac{\a q}{b},&\frac{gq}{\b}\\
&\frac{hq}{\b}\end{matrix}\, ;q, tq^{p+k+j}\bigg),\nonumber\\
Y_{23}&=\bigg(1-\frac{q}{\b}\bigg)\frac{(g,e, \g)_{\infty}}{(h, f, \d)_{\infty}}\sum_{p=1}^{\infty}\frac{(\frac{h}{g})_pg^p}{(q)_p}\sum_{j=0}^{\infty}\frac{(\frac{f}{e})_je^j}{(q)_j}\sum_{k=0}^{\infty}\frac{(\frac{\d}{\g})_k\g^k}{(q)_k}\sum_{m=0}^{\infty}\frac{(\frac{hq^p}{g})_m(tq^{p+k+j})_m}{(q^{p+1})_m(\frac{\a tq^{p+k+j}}{\b})_{m+1}}\left(\frac{gq}{\b}\right)^{m}.
\end{align}
Finally from \eqref{S}, \eqref{xeval}, \eqref{y12}, \eqref{y1}, \eqref{y2123}, and \eqref{y2123eval}, we arrive at \eqref{extandagarfur}.

\begin{proof}[Theorem \textup{\ref{newsigma2}}][]
Let $a=-z$ and $b=1$ in \eqref{4varrt} to obtain
\begin{equation*}
\rho_4(1,-z,c,d)=\rho_{4}(-z,1,c,d)-\frac{(d,c,-cd/z,-zq,-z^{-1},q)_{\infty}}{(d/z,-d,c/z,-c,zq,-q)_{\infty}}.
\end{equation*}
Divide both sides by $(1-z^{-1})$ and let $z\to 1$. Observe that using Lemma \ref{4heine}, the resulting right side is of the form $0/0$; hence employing L'Hopital's rule, we see that
\begin{align}\label{bbt}
&\sum_{n=0}^{\infty}\frac{(-cd)_n(1-cdq^{2n})q^{n(n+1)/2}}{(-q)_n(1-cq^n)(1-dq^n)}\nonumber\\
&=\e\left(2\sum_{n=0}^{\infty}\frac{(d,c,-cd/z)_n}{(zq)_n(-c, -d)_{n+1}}(1+cdq^{2n})z^nq^{n(n+1)/2}-\frac{(d,c,-cd/z,-zq,-z^{-1},q)_{\infty}}{(d/z,-d,c/z,-c,zq,-q)_{\infty}}\right)
\end{align}
The big task now is to transform the first series on the right side before applying the $\e$-operator. Note that
\begin{align}\label{bt}
&2\sum_{n=0}^{\infty}\frac{(d,c,-cd/z)_n}{(zq)_n(-c, -d)_{n+1}}(1+cdq^{2n})z^nq^{n(n+1)/2}\nonumber\\
&=\frac{2}{(1+c)(1+d)}\lim_{\tau\to 0}\sum_{n=0}^{\infty}\frac{(-\frac{q}{\tau},d,c,-cd/z)_n}{(\tau, zq,-cq, -dq)_{n}}(1+cdq^{2n})(\tau z)^n\nonumber\\
&=\frac{2}{(1+c)(1+d)}\left\{\lim_{\tau\to 0}\sum_{n=0}^{\infty}\frac{(-\frac{q}{\tau},d,c,-cd/z)_n}{(\tau, zq,-cq, -dq)_{n}}(\tau z)^n+cd\lim_{\tau\to 0}\sum_{n=0}^{\infty}\frac{(-\frac{q}{\tau},d,c,-cd/z)_n}{(\tau, zq,-cq, -dq)_{n}}(\tau zq^2)^n\right\}\nonumber\\
&=:\frac{2}{(1+c)(1+d)}(L_1+L_2).
\end{align}
To evaluate $L_1$, let $\a=-q/\tau, \b=zq, \g=c, \d=-cq, e=d, f=-dq, g=-cd/z, h=\tau$ and $t=\tau z$ in Theorem \ref{extandagar9}. This results in
{\allowdisplaybreaks\begin{align}\label{liml1}
L_1&=\frac{(-\frac{cd}{z}, d, c, q, -\frac{cq}{z}, -\frac{dq}{z}, -zq, -\frac{1}{z})_{\infty}}{(-dq, -cq, zq, -1, -q, \frac{c}{z}, \frac{d}{z}, -\frac{cd}{z^2})_{\infty}}\sum_{n=0}^{\infty}\frac{(\frac{c}{z},\frac{d}{z},-\frac{cd}{z^2})_n}{(-\frac{cq}{z},-\frac{dq}{z},q)_n}q^{\frac{n(n+1)}{2}}\nonumber\\
&\quad+\frac{(-\frac{cd}{z}, d, c, -\frac{cq}{z}, -\frac{dq}{z})_{\infty}}{(-dq, -cq, -1, \frac{c}{z}, \frac{d}{z}, -\frac{cd}{z^2})_{\infty}}\left(1-\frac{1}{z}\right)\sum_{n=0}^{\infty}\frac{(\frac{c}{z},\frac{d}{z},-\frac{cd}{z^2})_n}{(-\frac{cq}{z},-\frac{dq}{z},q)_n}q^{\frac{n(n+1)}{2}}\sum_{j=1}^{\infty}\frac{q^{j(j+1)/2}z^{-j}}{(-q)_j}\nonumber\\
&\quad+\frac{(-\frac{cd}{z}, d, c, -\frac{dq}{z})_{\infty}}{(-dq, -cq, -1, \frac{d}{z}, -\frac{cd}{z^2})_{\infty}}\left(1-\frac{1}{z}\right)\sum_{p=0}^{\infty}\frac{(-q)_p(-1)_pc^p}{(q)_p}\sum_{k=0}^{\infty}\frac{(-q^{p+1})_k(c/z)^k}{(q^{p+1})_k}\nonumber\\
&\quad\quad\times\sum_{n=0}^{\infty}\frac{(-\frac{cd}{z^2},\frac{d}{z})_n}{(-\frac{dq}{z},q)_n}q^{\frac{n(n+1)}{2}+(p+k)n}\nonumber\\
&\quad+\frac{(-\frac{cd}{z}, d, c)_{\infty}}{(-dq, -cq, -1, -\frac{cd}{z^2})_{\infty}}\left(1-\frac{1}{z}\right)\sum_{p=1}^{\infty}\frac{(-q)_p(-1)_pd^p}{(q)_p}\sum_{k=0}^{\infty}\frac{(-q)_k(-q^p)_kc^k}{(q)_k}\nonumber\\
&\quad\quad\times\sum_{j=0}^{\infty}\frac{(-q^{p+1})_j(d/z)^j}{(q^{p+1})_j}\sum_{n=0}^{\infty}\frac{(-\frac{cd}{z^2})_n}{(q)_n}q^{\frac{n(n+1)}{2}+(p+k+j)n}\nonumber\\
&\quad+\frac{(-\frac{cd}{z}, d, c)_{\infty}}{(-dq, -cq)_{\infty}}\left(1-\frac{1}{z}\right)\sum_{p=1}^{\infty}\frac{(-\frac{cd}{z})^p}{(q)_p}\sum_{j=0}^{\infty}\frac{(-q)_jd^j}{(q)_j}\sum_{k=0}^{\infty}\frac{(-q)_kc^k}{(q)_k}\sum_{m=0}^{\infty}\frac{(-cd/z^2)^m}{(q^{p+1})_m(-q^{p+k+j})_{m+1}}.
\end{align}}
Now let $\a=-q/\tau, \b=zq, \g=c, \d=-cq, e=d, f=-dq, g=-cd/z, h=\tau$ and $t=\tau zq^2$ in Theorem \ref{extandagar9}. This gives
\begin{align}\label{liml2}
L_2&=\frac{(-\frac{cd}{z}, d, c, q, -\frac{cq}{z}, -\frac{dq}{z}, -zq^3, -\frac{1}{zq^2})_{\infty}}{(-dq, -cq, zq, -\frac{1}{q^2}, -q^3, \frac{c}{z}, \frac{d}{z}, -\frac{cd}{z^2})_{\infty}}\sum_{n=0}^{\infty}\frac{(\frac{c}{z},\frac{d}{z},-\frac{cd}{z^2})_n}{(-\frac{cq}{z},-\frac{dq}{z},q)_n}q^{\frac{n(n+1)}{2}+2n}\nonumber\\
&\quad+\frac{(-\frac{cd}{z}, d, c, -\frac{cq}{z}, -\frac{dq}{z})_{\infty}}{(-dq, -cq, -q^2, \frac{c}{z}, \frac{d}{z}, -\frac{cd}{z^2})_{\infty}}\left(1-\frac{1}{z}\right)\sum_{n=0}^{\infty}\frac{(\frac{c}{z},\frac{d}{z},-\frac{cd}{z^2})_n}{(-\frac{cq}{z},-\frac{dq}{z},q)_n}q^{\frac{n(n+1)}{2}+2n}\sum_{j=1}^{\infty}\frac{q^{j(j-3)/2}z^{-j}}{(-1/q)_j}\nonumber\\
&\quad+\frac{(-\frac{cd}{z}, d, c, -\frac{dq}{z})_{\infty}}{(-dq, -cq, -q^2, \frac{d}{z}, -\frac{cd}{z^2})_{\infty}}\left(1-\frac{1}{z}\right)\sum_{p=0}^{\infty}\frac{(-q)_p(-q^2)_pc^p}{(q)_p}\sum_{k=0}^{\infty}\frac{(-q^{p+1})_k(c/z)^k}{(q^{p+1})_k}\nonumber\\
&\quad\quad\times\sum_{n=0}^{\infty}\frac{(-\frac{cd}{z^2},\frac{d}{z})_n}{(-\frac{dq}{z},q)_n}q^{\frac{n(n+1)}{2}+(p+k+2)n}\nonumber\\
&\quad+\frac{(-\frac{cd}{z}, d, c)_{\infty}}{(-dq, -cq, -q^2, -\frac{cd}{z^2})_{\infty}}\left(1-\frac{1}{z}\right)\sum_{p=1}^{\infty}\frac{(-q)_p(-q^2)_pd^p}{(q)_p}\sum_{k=0}^{\infty}\frac{(-q)_k(-q^{p+2})_kc^k}{(q)_k}\nonumber\\
&\quad\quad\times\sum_{j=0}^{\infty}\frac{(-q^{p+1})_j(d/z)^j}{(q^{p+1})_j}\sum_{n=0}^{\infty}\frac{(-\frac{cd}{z^2})_n}{(q)_n}q^{\frac{n(n+1)}{2}+(p+k+j+2)n}\nonumber\\
&\quad+\frac{(-\frac{cd}{z}, d, c)_{\infty}}{(-dq, -cq)_{\infty}}\left(1-\frac{1}{z}\right)\sum_{p=1}^{\infty}\frac{(-\frac{cd}{z})^p}{(q)_p}\sum_{j=0}^{\infty}\frac{(-q)_jd^j}{(q)_j}\sum_{k=0}^{\infty}\frac{(-q)_kc^k}{(q)_k}\sum_{m=0}^{\infty}\frac{(-cd/z^2)^m}{(q^{p+1})_m(-q^{p+k+j+2})_{m+1}}.
\end{align}
Substituting \eqref{liml1} and \eqref{liml2} in \eqref{bt}, we find that
{\allowdisplaybreaks\begin{align*}
&2\sum_{n=0}^{\infty}\frac{(d,c,-cd/z)_n}{(zq)_n(-c, -d)_{n+1}}(1+cdq^{2n})z^nq^{n(n+1)/2}\nonumber\\
&=\frac{2}{(1+c)(1+d)}\bigg\{\frac{(-\frac{cd}{z}, d, c, q, -\frac{cq}{z}, -\frac{dq}{z}, -zq, -\frac{1}{z})_{\infty}}{(-dq, -cq, zq, -1, -q, \frac{c}{z}, \frac{d}{z}, -\frac{cd}{z^2})_{\infty}}\sum_{n=0}^{\infty}\frac{(\frac{c}{z},\frac{d}{z},-\frac{cd}{z^2})_n}{(-\frac{cq}{z},-\frac{dq}{z},q)_n}\left(1+\frac{cdq^{2n}}{z^2}\right)q^{\frac{n(n+1)}{2}}\nonumber\\
&\quad+\frac{(-\frac{cd}{z}, d, c, -\frac{cq}{z}, -\frac{dq}{z})_{\infty}}{(-dq, -cq, -1, \frac{c}{z}, \frac{d}{z}, -\frac{cd}{z^2})_{\infty}}\left(1-\frac{1}{z}\right)\sum_{n=0}^{\infty}\frac{(\frac{c}{z},\frac{d}{z},-\frac{cd}{z^2})_n}{(-\frac{cq}{z},-\frac{dq}{z},q)_n}q^{\frac{n(n+1)}{2}}\nonumber\\
&\quad\quad\times\left\{\sum_{j=1}^{\infty}\frac{q^{j(j+1)/2}z^{-j}}{(-q)_j}+2(1+q)cdq^{2n}\sum_{j=1}^{\infty}\frac{q^{j(j-3)/2}z^{-j}}{(-1/q)_j}\right\}\nonumber\\
&\quad+\frac{(-\frac{cd}{z}, d, c, -\frac{dq}{z})_{\infty}}{(-dq, -cq, -1, \frac{d}{z}, -\frac{cd}{z^2})_{\infty}}\left(1-\frac{1}{z}\right)\sum_{p=0}^{\infty}\frac{(-q)_p(-1)_pc^p}{(q)_p}\sum_{k=0}^{\infty}\frac{(-q^{p+1})_k(c/z)^k}{(q^{p+1})_k}\nonumber\\
&\quad\quad\times\sum_{n=0}^{\infty}\frac{(-\frac{cd}{z^2},\frac{d}{z})_n}{(-\frac{dq}{z},q)_n}q^{\frac{n(n+1)}{2}+(p+k)n}\left(1+cdq^{2n}(1+q^p)(1+q^{p+1})\right)\nonumber\\
&\quad+\frac{(-\frac{cd}{z}, d, c)_{\infty}}{(-dq, -cq, -1, -\frac{cd}{z^2})_{\infty}}\left(1-\frac{1}{z}\right)\sum_{p=1}^{\infty}\frac{(-q)_p(-1)_pd^p}{(q)_p}\sum_{k=0}^{\infty}\frac{(-q)_k(-q^{p})_kc^k}{(q)_k}\nonumber\\
&\quad\quad\times\sum_{j=0}^{\infty}\frac{(-q^{p+1})_j(d/z)^j}{(q^{p+1})_j}\sum_{n=0}^{\infty}\frac{(-\frac{cd}{z^2})_n}{(q)_n}q^{\frac{n(n+1)}{2}+(p+k+j)n}\left(1+cdq^{2n}(1+q^{p+k})(1+q^{p+k+1})\right)\nonumber\\
&\quad+\frac{(-\frac{cd}{z}, d, c)_{\infty}}{(-dq, -cq)_{\infty}}\left(1-\frac{1}{z}\right)\sum_{p=1}^{\infty}\frac{(-\frac{cd}{z})^p}{(q)_p}\sum_{j=0}^{\infty}\frac{(-q)_jd^j}{(q)_j}\sum_{k=0}^{\infty}\frac{(-q)_kc^k}{(q)_k}\nonumber\\
&\quad\quad\times\sum_{m=0}^{\infty}\frac{(-cd/z^2)^m}{(q^{p+1})_m(-q^{p+k+j})_{m+1}}\left(1+cd\frac{(1+q^{p+k+j})(1+q^{p+k+j+1})}{(1+q^{p+k+j+m+1})(1+q^{p+k+j+m+2})}\right)\bigg\}.
\end{align*}}
Since
\begin{align*}
&\sum_{n=0}^{\infty}\frac{(\frac{c}{z},\frac{d}{z},-\frac{cd}{z^2})_n}{(-\frac{cq}{z},-\frac{dq}{z},q)_n}\left(1+\frac{cdq^{2n}}{z^2}\right)q^{\frac{n(n+1)}{2}}=\frac{\left(-\frac{cd}{z^2},-q\right)_{\infty}}{\left(-\frac{cq}{z},-\frac{dq}{z}\right)_{\infty}},
\end{align*}
by Lemma \ref{4heine}, and
\begin{align*}
&\sum_{j=1}^{\infty}\frac{q^{j(j+1)/2}z^{-j}}{(-q)_j}+2(1+q)cdq^{2n}\sum_{j=1}^{\infty}\frac{q^{j(j-3)/2}z^{-j}}{(-1/q)_j}\nonumber\\
&=\left(1+\frac{cdq^{2n}}{z^2}\right)\sum_{j=1}^{\infty}\frac{q^{j(j+1)/2}z^{-j}}{(-q)_j}+\frac{cdq^{2n}}{z^2}(1+2z),
\end{align*}
we see that
{\allowdisplaybreaks\begin{align}\label{bts2}
&2\sum_{n=0}^{\infty}\frac{(d,c,-cd/z)_n}{(zq)_n(-c, -d)_{n+1}}(1+cdq^{2n})z^nq^{n(n+1)/2}\nonumber\\
&=\frac{(-\frac{cd}{z}, d, c, q, -zq, -\frac{1}{z})_{\infty}}{(-d, -c, zq, -q, \frac{c}{z}, \frac{d}{z})_{\infty}}+\frac{(-\frac{cd}{z}, d, c)_{\infty}}{(-d, -c, \frac{c}{z}, \frac{d}{z})_{\infty}}\left(1-\frac{1}{z}\right)\sum_{j=1}^{\infty}\frac{q^{j(j+1)/2}z^{-j}}{(-q)_j}\nonumber\\
&\quad+\frac{cd(1+2z)}{z^2}\left(1-\frac{1}{z}\right)\frac{(-\frac{cd}{z}, d, c, -\frac{cq}{z}, -\frac{dq}{z})_{\infty}}{(-d, -c, -q, \frac{c}{z}, \frac{d}{z}, -\frac{cd}{z^2})_{\infty}}\sum_{n=0}^{\infty}\frac{(\frac{c}{z},\frac{d}{z},-\frac{cd}{z^2})_n}{(-\frac{cq}{z},-\frac{dq}{z},q)_n}q^{\frac{n(n+1)}{2}+2n}\nonumber\\
&\quad+\frac{(-\frac{cd}{z}, d, c, -\frac{dq}{z})_{\infty}}{(-d, -c, -q, \frac{d}{z}, -\frac{cd}{z^2})_{\infty}}\left(1-\frac{1}{z}\right)\sum_{p=0}^{\infty}\frac{(-q)_p(-1)_pc^p}{(q)_p}\sum_{k=0}^{\infty}\frac{(-q^{p+1})_k(c/z)^k}{(q^{p+1})_k}\nonumber\\
&\quad\quad\times\sum_{n=0}^{\infty}\frac{(-\frac{cd}{z^2},\frac{d}{z})_n}{(-\frac{dq}{z},q)_n}q^{\frac{n(n+1)}{2}+(p+k)n}\left(1+cdq^{2n}(1+q^p)(1+q^{p+1})\right)\nonumber\\
&\quad+\frac{(-\frac{cd}{z}, d, c)_{\infty}}{(-d, -c, -q, -\frac{cd}{z^2})_{\infty}}\left(1-\frac{1}{z}\right)\sum_{p=1}^{\infty}\frac{(-q)_p(-1)_pd^p}{(q)_p}\sum_{k=0}^{\infty}\frac{(-q)_k(-q^{p})_kc^k}{(q)_k}\nonumber\\
&\quad\quad\times\sum_{j=0}^{\infty}\frac{(-q^{p+1})_j(d/z)^j}{(q^{p+1})_j}\sum_{n=0}^{\infty}\frac{(-\frac{cd}{z^2})_n}{(q)_n}q^{\frac{n(n+1)}{2}+(p+k+j)n}\left(1+cdq^{2n}(1+q^{p+k})(1+q^{p+k+1})\right)\nonumber\\
&\quad+2\frac{(-\frac{cd}{z}, d, c)_{\infty}}{(-d, -c)_{\infty}}\left(1-\frac{1}{z}\right)\sum_{p=1}^{\infty}\frac{(-\frac{cd}{z})^p}{(q)_p}\sum_{j=0}^{\infty}\frac{(-q)_jd^j}{(q)_j}\sum_{k=0}^{\infty}\frac{(-q)_kc^k}{(q)_k}\nonumber\\
&\quad\quad\times\sum_{m=0}^{\infty}\frac{(-cd/z^2)^m}{(q^{p+1})_m(-q^{p+k+j})_{m+1}}\left(1+cd\frac{(1+q^{p+k+j})(1+q^{p+k+j+1})}{(1+q^{p+k+j+m+1})(1+q^{p+k+j+m+2})}\right).
\end{align}}
Finally we substitute \eqref{bts2} in \eqref{bbt} and then apply the $\e$-operator to obtain \eqref{idsigma2} after simplification. This completes the proof.
\end{proof}
\textbf{Remark 2.} Let $S(c, d,q):=(-cd,-q)_{\infty}/(-d,-c)_{\infty}$ and denote the left-hand side of \eqref{bbt} by $\sigma(c, d, q)$. If we explicitly evaluate $\e\left(\frac{(-cd/z,d,c,q,-zq,-z^{-1},q)_{\infty}}{(-d,-c,zq,-q,c/z,d/z)_{\infty}}\right)$ using the Jacobi triple product identity, then \eqref{bbt} leads us to
\begin{align*}
\sigma(c, d, q)&=\e\left(2\sum_{n=0}^{\infty}\frac{(d,c,-cd/z)_n(1+cdq^{2n})}{(zq)_n(-c, -d)_{n+1}}z^nq^{\frac{n(n+1)}{2}}\right)+S(c, d, q)\left(1+2\sum_{n=0}^{\infty}\frac{cdq^{n}}{1+cdq^{n}}\right)\nonumber\\
&\quad+2S(c, d,q)\left(\sum_{n=0}^{\infty}\frac{dq^n}{1-dq^n}+\sum_{n=0}^{\infty}\frac{cq^n}{1-cq^n}-\sum_{n=1}^{\infty}\frac{q^n}{1-q^n}\right).
\end{align*}
This gives a two-variable generalization of \cite[Equation (3.5)]{andrews1986}, as can be seen with the help of \eqref{rsi1}.

\subsection{Theorem \ref{newsigma1} as a special case of Theorem \ref{newsigma2}}\label{spcor}
In this subsection, we deduce Theorem \ref{newsigma1} from Theorem \ref{newsigma2}. 

Let $d=0$ in \eqref{idsigma2}. Then
\begin{align*}
\sigma(q)=(-c)_{\infty}\sum_{n=0}^{\infty}\frac{q^{n(n+1)/2}}{(-q)_n(1-cq^n)}+\Lambda(c, 0, q).
\end{align*}
Thus we need only show that 
\begin{equation*}
\Lambda(c, 0, q)=-2\sum_{m,n=0}^{\infty}\frac{(-q)_m}{(q)_m(q)_n}\frac{(-1)^nq^{n(n+1)/2}c^{m+n+1}}{(1-q^{n+m+1})}.
\end{equation*}
To that end, note that the two quadruple sums in \eqref{Lambda} just collapse to $0$ so that
\begin{equation*}
\Lambda(c, 0, q)=1-\frac{(c)_{\infty}}{(-q)_{\infty}}\sum_{p=0}^{\infty}\frac{(-q)_p(-1)_pc^p}{(q)_p}\sum_{k=0}^{\infty}\frac{(-q^{p+1})_kc^k}{(q^{p+1})_k}\sum_{n=0}^{\infty}\frac{q^{\frac{n(n+1)}{2}+(p+k)n}}{(q)_n}.
\end{equation*}
 Now use Euler's formula \cite[p.~19, Corollary 2.2]{gea} 
\begin{equation*}
\sum_{n=0}^{\infty}\frac{w^nq^{\frac{n(n-1)}{2}}}{(q)_n}=(-w)_{\infty},\hspace{3mm} |w|<\infty, 
\end{equation*}
to evaluate the sum over $n$ in the above triple sum so that
{\allowdisplaybreaks\begin{align*}
\Lambda(c, 0, q)&=1-\frac{(c)_{\infty}}{(-q)_{\infty}}\sum_{p=0}^{\infty}\frac{(-q)_p(-1)_pc^p}{(q)_p}\sum_{k=0}^{\infty}\frac{(-q^{p+1})_k}{(q^{p+1})_k}c^k(-q^{p+k+1})_{\infty}\nonumber\\
&=1-(c)_{\infty}\sum_{k=0}^{\infty}\frac{c^k}{(q)_k}-(c)_{\infty}\sum_{p=1}^{\infty}\frac{(-1)_p}{(q)_p}c^p\sum_{k=0}^{\infty}\frac{c^k}{(q^{p+1})_k}\nonumber\\
&=-2\sum_{p,n=0}^{\infty}\frac{(-q)_p}{(q)_p(q)_n}\frac{(-1)^nq^{n(n+1)/2}c^{p+n+1}}{(1-q^{p+n+1})},
\end{align*}}
where we used \eqref{qbin} to evaluate the first sum in the penultimate expression, and \eqref{1551fine} to evaluate the sum over $k$ in the double sum over $p$ and $k$. 
This completes the proof.

\section{Concluding Remarks}\label{cr}
The two series, namely 
\begin{align*}
\sum_{n=0}^{\infty}\frac{q^{n(n+1)/2}}{(-q)_n(1-cq^n)},\hspace{2mm}\text{and}\hspace{2mm}\sum_{n=0}^{\infty}\frac{(-cd)_n(1-cdq^{2n})q^{n(n+1)/2}}{(-q)_n(1-cq^n)(1-dq^n)},
\end{align*}
occurring in theorems \ref{newsigma1} and \ref{newsigma2}, are respectively one- and two-variable generalizations of $\sigma(q)$. It may be fruitful to see which properties of $\sigma(q)$ hold for these generalizations as well. Also it may be important to see if there are any partition-theoretic interpretations of some of the results proved here. 

As demonstrated in this paper, there are a number of advantages of using Agarwal's identity \eqref{me} and its generalization \eqref{extandagarfur} for transforming $\rho_{3}(-z,1,c)$ and $\rho_4(-z,1,c,d)$ respectively. First of all, the infinite product expressions occurring in the specializations of the three- and the four-variable reciprocity theorems used in our proofs get cancelled completely. Secondly, these identities contain ${}_{2}\phi_{1}\bigg(\begin{matrix}q,& q/t\\
&q\b/(\a t)\end{matrix}\, ;q, q/\a\bigg)$, which is what leads to $\sigma(q)$ after appropriately specializing the parameters. Thirdly, all of the other expressions in these identities contain the factor $1-q/\b$, or after letting $\b=zq$, the factor $1-1/z$, which is extremely useful since all other factors involving $z$ in an expression which contains $1-1/z$ get annihilated when we differentiate them with respect to $z$ and then let $z\to 1$.

There are further generalizations of Ramanujan's reciprocity theorem, namely, the five-variable generalization due to Chu and Zhang \cite{chuzhang} and Ma \cite[Theorem 1.3]{masix}, the six-variable generalization given in \cite{masix}, the seven-variable generalization due to Wei, Wang and Yan \cite[Theorem 3, Corollary 4]{wwy} and a different one by Liu \cite[Theorem 1.9]{zgl2016}, and finally the multiparameter generalization in \cite[Theorem 7]{wwy}. While there is no reason a priori why the ideas used in this paper  may not be applicable to obtain further identities of the type we have established,  the complexity of the computations involved in the proof of Theorem \ref{newsigma2} suggests that the computations involved while applying the reciprocity theorems in more than four variables may be quite unwieldy. 

That being said, we believe that one can further simplify $\Lambda(c, d, q)$ to the effect of at least having the $1$ on the right-hand side of \eqref{Lambda} cancelled. First of all, note that the second expression in \eqref{Lambda} admits further simplification, namely,
{\allowdisplaybreaks\begin{align*}
&\frac{3cd(-cq,-dq)_{\infty}}{(-cd,-q)_{\infty}}\sum_{n=0}^{\infty}\frac{(c,d,-cd)_n}{(-cq,-dq,q)_n}q^{\frac{n(n+1)}{2}+2n}\nonumber\\
&=\frac{3(-cq,-dq)_{\infty}}{(-cd,-q)_{\infty}}\left\{\sum_{n=0}^{\infty}\frac{(c,d,-cd)_n}{(-cq,-dq,q)_n}(1+cdq^{2n})q^{\frac{n(n+1)}{2}}-\sum_{n=0}^{\infty}\frac{(c,d,-cd)_n}{(-cq,-dq,q)_n}q^{\frac{n(n+1)}{2}}\right\}\nonumber\\
&=3-3\frac{(-cq,-dq)_{\infty}}{(-cd,-q)_{\infty}}\sum_{n=0}^{\infty}\frac{(c,d,-cd)_n}{(-cq,-dq,q)_n}q^{\frac{n(n+1)}{2}},
\end{align*}}
by another application of Lemma \ref{4heine}. However, we are unable to evaluate the last series or the other multi-sums occurring in \eqref{Lambda}. Note that the following special case of the $q$-analog of Kummer's theorem \cite[Equation (1.7)]{gea55}, known as Lebesgue's identity, is well-known \cite[Corollary 2.7]{gea}:
\begin{equation*}
\sum_{n=0}^{\infty}\frac{(a)_n}{(q)_n}q^{n(n+1)/2}=(-q)_{\infty}(aq;q^2)_{\infty}.
\end{equation*}
This prompts us to ask if there are higher-level analogues of Lebesgue's identity which could possibly be used to represent the sum $\displaystyle\sum_{n=0}^{\infty}\frac{(c,d,-cd)_n}{(-cq,-dq,q)_n}q^{\frac{n(n+1)}{2}}$. A generalization of Lebesgue's identity in a different direction is given by Alladi \cite[Equation (2.10), Section 4]{alladileg}.

The finite forms of Ramanujan's reciprocity theorem and its three- and four-variable generalizations are obtained in \cite{sommur}. It may be of interest to see if something along the lines of \eqref{rsi1}, \eqref{rsi2}, and Theorems \ref{newsigma1} and \ref{newsigma2} could be obtained starting with these finite analogues.

\begin{center}
\textbf{Acknowledgements}
\end{center}
This work was done when the first author was visiting Indian Institute of Technology Gandhinagar as a summer intern under the Summer Research Internship Program-2016 (SRIP-2016) program. He would like to thank the institute for its hospitality and support.


\begin{thebibliography}{00}


\bibitem{agar1}
R.P.~Agarwal, \emph{On the paper ``A `Lost' Notebook of Ramanujan''}, Adv. Math.~\textbf{53} (1984), 291--300.

\bibitem{alladileg}
K.~Alladi, \emph{Partitions with non-repeating odd parts and $q$-hypergeometric identities}, \emph{The legacy of Alladi Ramakrishnan in the mathematical sciences}, K. Alladi, J. Klauder, C. R. Rao, Eds., Springer, New York, 2010, pp. 169--182.

\bibitem{gea55} 
G.E.~Andrews, \emph{On the $q$-analog of Kummer's theorem and applications}, Duke Math. J.~\textbf{40} (1973), 525--528.

\bibitem{gea}
G.E.~Andrews, The Theory of Partitions, Addison-Wesley Pub. Co., NY, 300 pp. (1976). Reissued, Cambridge University Press, New York, 1998.

\bibitem{gea90}
G.E.~Andrews, \emph{Ramanujan's ``Lost" notebook: \textup{I}. Partial theta functions}, Adv. Math.~\textbf{41} (1981), 137--172.

\bibitem{andrews1986}
G.E.~Andrews, \emph{Ramanujan's ``Lost'' Notebook V: Euler's partition identity}, Adv.~Math.~\textbf{61} (1986), 156--164.

\bibitem{andrewsmonthly86}
G.E.~Andrews, \emph{Questions and conjectures in partition theory}, Amer. Math. Monthly~\textbf{93} (1986), 708--711.

\bibitem{ablnb2}
G.E.~Andrews and B.C.~Berndt, \emph{Ramanujan's Lost Notebook. Part II}, Springer Science and Business Media, 2009.

\bibitem{adh}
G.E.~Andrews, F.J.~Dyson and D.~Hickerson, \emph{Partitions and indefinite quadratic forms}, Invent.~Math.~\textbf{91} (1988), 391--407.

\bibitem{ajo}
G.E.~Andrews, J.~Jim\'{e}nez-Urroz and K.~Ono, \emph{$q$-series identities and values of certain $L$-functions}, Duke Math. J.~\textbf{108} No. 3 (2001), 395--419.

\bibitem{ady1}
G.E.~Andrews, A.~Dixit and A.J.~Yee, \emph{Partitions associated with the Ramanujan/Watson mock theta functions $\omega(q), \nu(q)$ and $\phi(q)$}, \emph{Research in Number Theory},~\textbf{1} Issue 1 (2015), 1--25.

\bibitem{adsy1}
G.E.~Andrews, A.~Dixit, D.~Schultz and A.J.~Yee, \emph{Overpartitions related to the mock theta function $\omega(q)$}, submitted for publication.

\bibitem{af}
G.E.~Andrews and P.~Freitas, \emph{Extension of Abel's Lemma with q-series implications}, Ramanujan J.~\textbf{10} (2005), 137–-152.

\bibitem{bringkane1}
K.~Bringmann and B.~Kane, \emph{Multiplicative q-hypergeometric series arising from real quadratic fields}, Trans. Amer. Math. Soc.~\textbf{363} No. 4 (2011), 2191--2209.

\bibitem{bringkane2}
K.~Bringmann and B.~Kane, \emph{New identities involving sums of the tails related to real quadratic fields}, Ramanujan J.~\textbf{23} (2010), 243--251 (Special issue in honor of George E. Andrews' $70$th birthday).

\bibitem{cflz}
D.~Corson, D.~Favero, K.~Liesinger and S.~Zubairy, \emph{Characters and $q$-series in $\mathbb{Q}(\sqrt{2})$}, J.~Number Theory~\textbf{107} (2004), 392--405.


\bibitem{chansears}
S.H.~Chan, \emph{On Sears's general transformation formula for basic hypergeometric series}, Ramanujan J.~\textbf{20} (2009), 69--79.

\bibitem{chenji}
W.Y.C.~Chen and K.Q.~Ji, \emph{Weighted forms of Euler's theorem}, J.~Combin. Theory Ser. A~\textbf{114} No. 2 (2007), 360--372.

\bibitem{chuzhang}
W.~Chu and W.~Zhang, \emph{Bilateral $q$-series identities and reciprocal formulae}, Funct.~Approximatio, Comment. Math.~\textbf{42} No. 2 (2010), 153--162.

\bibitem{cohen}
H.~Cohen, \emph{$q$-Identities for Maass waveforms}, Invent. Math.~\textbf{91} (1988), 409--422.

\bibitem{fine}
N.J. Fine, {Basic hypergeometric series and applications}, Amer. Math. Soc., Providence, 1988.

\bibitem{mgupta}
M.~Gupta, \emph{A Study of Certain Basic Hypergeometric Identities and their Applications}, Doctoral Thesis, University of Rajasthan, Jaipur, 1989.

\bibitem{kang}
S.-Y.~Kang, \emph{Generalizations of Ramanujan's reciprocity theorem and their applications}, J.~London Math.~Soc.~\textbf{75} (2) (2007), 18--34.

\bibitem{krw}
M.~Krauel, L.~Rolen and M.~Woodbury, \emph{On a relation between certain $q$-hypergeometric series and Maass waveforms} arXiv: 1512.04452v2, March 15, 2016.
\url{http://arxiv.org/pdf/1512.04452.pdf}
 
\bibitem{lingorhoades1}
Y.~Li, H.T.~Ngo and R.C.~Rhoades, \emph{Renormalization and quantum modular forms, Part I: Maass wave forms}, arXiv: 1311.3043v1, November 13, 2013.
\url{http://arxiv.org/pdf/1311.3043v1.pdf}


\bibitem{zgl2016}
Z.-G.~Liu, \emph{Extensions of Ramanujan's reciprocity theorem and the Andrews-Askey integral}, J.~Math.~Anal.~Appl.~\textbf{443} (2016), 1110--1129.
 
\bibitem{love1}
J.~Lovejoy, \emph{Lacunary partition functions}, Math. Res. Lett.~\textbf{9} (2002), 191-–198.

\bibitem{love2}
J.~Lovejoy, \emph{Overpartitions and real quadratic fields}, J. Number Theory~\textbf{106} (2004), 178-–186.

\bibitem{masix}
X.R.~Ma, \emph{Six-variable generalization of Ramanujan’s reciprocity theorem}, J. Math. Anal. Appl.~\text{353} (2009), 320–-328.

\bibitem{pat0}
A.~Patkowski, \emph{A family of lacunary partition functions}, New Zealand J.~Math.~\textbf{38} (2008), 87--91.

\bibitem{pat1}
A.~Patkowski, \emph{On curious generating functions for values of $L$-functions}, Int.~J.~Number Theory~\textbf{6} No. 7 (2010), 1531--1540.

\bibitem{pat2}
A.~Patkowski, \emph{An Observation on the extension of Abel’s Lemma}, Integers~\textbf{10} (2010), 793--800.

\bibitem{pat3}
A.~Patkowski, \emph{On the $q$-Pell sequences and sums of tails}, arXiv:1312.0134v2, December 8, 2013.
\url{https://arxiv.org/pdf/1312.0134v2.pdf}

\bibitem{lnb}
S.~Ramanujan, \emph{The Lost Notebook and Other Unpublished
Papers}, Narosa, New Delhi, 1988.

\bibitem{sommur}
D.D.~Somashekara and K.~Narasimha Murthy, \emph{Finite forms of reciprocity theorem of Ramanujan and its generalizations}, International J. Math. Combin.~\textbf{4} (2013), 1--14.

\bibitem{xiong}
X.~Xiong, \emph{Small values of coefficients of a half Lerch sum}, arXiv:1605.09508v1, May 31, 2016, 
\url{http://arxiv.org/abs/1605.09508v1}

\bibitem{wwy}
C.~Wei, X.~Wang and Q.~Yan, \emph{Generalizations of Ramanujan's reciprocity formula and the Askey-Wilson integral}, Ramanujan J.~\textbf{37}, No. 1 (2015), 203--217.

\bibitem{zagiertop}
D.~Zagier, \emph{Vassiliev invariants and a strange identity related to the Dedekind eta-function}, Topology~\textbf{40} (2001), 945--960.

\bibitem{zagierqmf}
D.~Zagier, \emph{Quantum modular forms}, Quantas of Maths, Clay Math. Proc., Vol. 11, Amer. Math. Soc., Providence, RI, 2010, pp. 659--675.
\end{thebibliography}
\end{document}